\documentclass[twoside, 12pt]{article}
\usepackage{amsmath, amsthm, amssymb, amscd, mathptmx}
\usepackage{amsfonts}

\usepackage{pstricks}

\usepackage[all]{xy}\CompileMatrices\SelectTips{cm}{12}

\theoremstyle{plain}
\newtheorem{Thm}{\sc Theorem}[section]
\newtheorem{Theorem}[Thm]{\sc Theorem}
\newtheorem{Corollary}[Thm]{\sc Corollary}
\newtheorem*{Corollary*}{\sc Corollary}
\newtheorem{Proposition}[Thm]{\sc Proposition}
\newtheorem*{Proposition*}{\sc Proposition}
\newtheorem{Lemma}[Thm]{\sc Lemma}
\newtheorem{Conjecture}[Thm]{\sc Conjecture}
\newtheorem{Question}[Thm]{\sc Question}

\theoremstyle{definition}

\theoremstyle{remark}
\newtheorem{Remark}[Thm]{Remark}
\newtheorem{Remarks}[Thm]{Remarks}
\newtheorem{Example}[Thm]{Example}
\newtheorem*{Example*}{Example}
\newtheorem*{Remark*}{Remark}

\newcommand{\cX}{{\cal X}}

\newcommand{\Tr}{\mathop{\rm Tr}}

\newcommand{\id}{{\mathop{\rm id}}}
\newcommand{\ZZ}{{\mathbb Z}}
\newcommand{\FF}{{\mathbb F}}

\newcommand{\Spec}{\mathop{\rm Spec \, }}

\newcommand{\cO}{{\mathcal O}}
\newcommand{\cA}{{\cal A}}
\newcommand{\cE}{{\cal E}}
\newcommand{\cF}{{\cal F}}

\renewcommand{\AA}{{\mathbb A}}
\newcommand{\HH}{{\mathbb H}}
\newcommand{\NN}{{\mathbb N}}

\newcommand{\PP}{{\mathbb P}}
\newcommand{\QQ}{{\mathbb Q}}
\newcommand{\RR}{{\mathbb R}}
\newcommand{\CC}{{\mathbb C}}

\newcommand{\cL}{{\cal L}}

\newcommand{\ti}{\tilde}

\newcommand{\Hom}{{\mathop{{\rm Hom}}}}

\newcommand{\End}{{\mathop{ End}\,}}

\newcommand{\Ext}{{\mathop{{\rm E}xt}}}

\newcommand{\GL}{\mathop{\rm GL}}

\newcommand{\rk}{\mathop{{rk}}}

\newcommand{\NE}{\mathop{\rm NE}}

\mathchardef\mhyp="2D

\begin{document}

\title{On positivity and semistability of vector bundles in finite and mixed characteristics}
\author{Adrian Langer}
\date{\today}

\maketitle


{{\sc Address:}\\
Institute of Mathematics, Warsaw University, ul.\ Banacha 2,
02-097 Warszawa, Poland\\}

\bigskip
\emph{Dedicated to Professor C. S. Seshadri on his 20th birthday}
\footnote{In this rare case the number of years does not coincide
with the number of birthdays.}

\begin{abstract}
We survey results concerning behavior of positivity of line
bundles and possible vanishing theorems in positive
characteristic. We also try to describe variation of positivity in
mixed characteristic. These problems are very much related to
behavior of strong semistability of vector bundles, which is
another main topic of the paper.
\end{abstract}

\section*{Introduction}

The main aim of this paper is to survey problems concerning
positivity of line bundles and stability of vector bundles on
schemes defined over finite fields or over finitely generated
rings over $\ZZ$. Note that these two topics are very much related
because a degree zero vector bundle $E$ on a curve is strongly
semistable if and only if the line bundle $\cO_{\PP (E)}(1)$ on
the projectivization of $E$ is nef (see, e.g., \cite[Proposition
7.1]{Mo}).

The motivating problems are the following:

\begin{itemize}
\item What can we say about relation between nefness,
semiampleness, effectivity and pseudoeffectivity for line bundles on
varieties defined over finite fields?
\item What vanishing theorems can hold for suitably positive line bundles
in positive characteristic (or over $\bar\FF _p$)?
\item Is there any relation between nefness in characteristic
zero and in positive characteristic?
\item What can we say about variation in families of positivity of line bundles
and semistability of vector bundles?
\end{itemize}

The known results do not answer any of these questions. In this
paper we pose and study some conjectures that try to answer all of
the above questions. Some of these question are very arithmetic in
nature and in fact they imply very strong properties of reductions
of varieties. In some simple cases they can be recovered using
known results or they give another point of view on well known
conjectures from arithmetic algebraic geometry.

\medskip

The paper is divided in several sections describing each of these
problems and surveying known results. First we recall some
notation used throughout the paper. In Section 1 we describe
positivity of line bundles on varieties defined over finite
fields.  In Section 2 we survey known results on Kodaira type
vanishing theorems in positive characteristic. In Section 3 we
study vanishing theorems for general reductions from
characteristic zero. In Section 4 we recall several known
constructions of strictly nef line bundles in characteristic zero.
This is related to Keel's question of existence of such bundles
over finite fields. In Section 5 we study variation of positivity
of line bundles in mixed characteristic. In Section 6 we consider
a related question concerning vector bundles. In both
Sections 6 and 7 we pose several conjectures that should fully
explain behavior of strong semistability in mixed characteristic.

\medskip

\subsection{Notation}

Let $X$ be a complete variety defined over some algebraically
closed field $k$.

Let $N_1(X)$ ($N^1(X)$)  be the group of $1$-cycles (divisors,
respectively) modulo numerical equivalence. By the N\'eron-Severi
theorem $N_1(X)_{\QQ}=N_1(X)\otimes \QQ$ and
$N^1(X)_{\QQ}=N^1(X)\otimes \QQ$ are finite dimensional
$\QQ$-vector spaces, dual to each other by the intersection
pairing.

A $\QQ$-divisor $D$ is called \emph{pseudoeffective} if its
numerical class in $N^1(X)_{\QQ}$ is contained in the closure of
the cone generated by the classes of effective divisors.

A line bundle $L$ on $X$ is called \emph{semiample}, if there
exists a positive integer $n$ such that $L^{\otimes n}$ is
globally generated.

A line bundle $L$ on a variety $X$ is called \emph{strictly nef}
if it has positive degree on every curve in $X$.

A locally free sheaf $E$ on $X$ is \emph{nef} if and only if for
any $k$-morphism $f: C\to X$ from a smooth projective curve $C/k$
each quotient of $f^*E$ has a non-negative degree. We say that $E$
is \emph{numerically flat} if both $E$ and $E^*$ are nef.

Let $X$ be a normal projective $k$-variety and let $H$ be an ample
Cartier divisor on $X$. Let $E$ be a rank $r$ torsion free sheaf
on $X$. Then we define the \emph{slope} $\mu _H(E)$ of $E$ as
quotient of the degree of $\det E=(\bigwedge ^r E)^{**}$ with
respect to $H$ by the rank $r$.

We say that $E$ is \emph{slope $H$-semistable} if for every
subsheaf $E'\subset E$ we have $\mu _H(E')\le \mu _H(E)$.

If $k$ has positive characteristic then we say that $E$ is
\emph{strongly slope $H$-semistable} if all the Frobenius pull
backs $(F_X^n)^*E$ of $E$ for $n\ge 0$ are slope $H$-semistable.

Let $X$ be an algebraic $k$-variety. We say that a \emph{very
general point} of $X$ satisfies some property if there exists a
countable union of proper subvarieties of $X$ such that the
property is satisfied for all points outside of this union.

\section{Nef line bundles over finite fields}

The following fact (see, e.g., \cite[Lemma 2.16]{Ke1}) is standard
and it follows easily from existence of the Picard scheme and the
fact that an abelian variety has only finitely many rational
points over a given finite field.

\begin{Proposition}\label{torsion-lb}
A numerically trivial line bundle on a projective scheme defined
over $\bar \FF _p$ is torsion. In particular, a nef line bundle on
a projective curve over $\bar \FF _p$ is semiample.
\end{Proposition}

In the surface case Artin \cite[2-2.11]{Ar} proved the following
result:

\begin{Theorem} \label{Artin}
A nef and big line bundle on a smooth projective surface defined
over $\bar\FF _p$ is semiample.
\end{Theorem}

In \cite[Theorem 0.2]{Ke1} Sean Keel gave the following criterion
for semiampleness:

\begin{Theorem}\label{Keel}
Let $L$ be a nef line bundle on a projective scheme $X$ defined
over a field of positive characteristic. Let $L^{\perp}$ be the
closure of the union of all subvarieties $Y\subset X$ such that
$L^{\dim Y}\cdot Y=0$, taken with the reduced scheme structure.
Then $L$ is semiample if and only if its restriction to
$L^{\perp}$ is semiample.
\end{Theorem}

This theorem, combined with earlier ideas of Seshadri, occurred to
be the main new ingredient in Seshadri's new proof of Mumford's
conjecture (see \cite{Se}).

A basic tool used in proofs of Theorems \ref{Artin} and \ref{Keel}
is Proposition \ref{torsion-lb}.

Keel's theorem implies Artin's theorem, because if $X/\bar \FF _p$
is a smooth projective surface and $L$ is a nef and big line
bundle on $X$ then $L^{\perp}$ is at most one-dimensional and
hence $L|_{L^{\perp}}$ is numerically trivial. Thus by Proposition
\ref{torsion-lb} $L|_{L^{\perp}}$ is torsion and Theorem
\ref{Keel} implies that $L$ is semiample.

\medskip

Note that Keel's theorem trivially fails in the characteristic
zero case. As an example one can take, e.g.,  any non-torsion line
bundle of degree zero on a smooth projective curve. It is more
difficult to produce counterexamples to Artin's theorem in the
characteristic zero case but they also exist:

\begin{Theorem} \emph{(see \cite[Theorem 3.0]{Ke1})}\label{Keel-example}
Let $C$ be a smooth projective curve of genus $g\ge 2$ over a
field of characteristic zero. Let $X=C\times C$ and let
$L=p_1^*\omega_C(\Delta)$, where $\Delta$ is the diagonal and
$p_1$ is the projection of $X$ onto the first factor. Then $L$ is
nef and big but it is not semiample.
\end{Theorem}

Note that in positive characteristic the bundle $L$ in the above
theorem is semiample. All these results and lack of good
construction methods raised the question whether there exist any
nef line bundles on varieties defined over finite fields which are
not semiample. In \cite[Section 5]{Ke2} Keel gives Koll\'ar's
example of a nef but non-semiample line bundle on a non-normal
surface defined over a finite field. The example is obtained by
glueing two copies of $\PP ^1\times \PP ^1$ but the obtained line
bundle is not strictly nef.

Keel's proof of non-semiampleness in Theorem \ref{Keel-example}
goes via showing that the restriction of $L$ to $2\Delta$ is
non-torsion. Interestingly, Totaro used a similar strategy to show
the following example of a nef but non-semiample line bundle on a
smooth projective surface over $\bar \FF _p$:

\begin{Example}\label{Tot-example}
Let $C$ be a smooth projective curve of genus $2$ defined over
$\bar \FF _p$. Assume that for every line bundle $L$ of order $\le
p$ the map $H^1(C, L)\to H^1(C, F^*L)$, induced by the Frobenius
morphism on C, is injective. In \cite[Lemma 6.4]{To} Totaro showed
that a general curve of genus $2$ satisfies this assumption.

Then one can embedd $C$ into $\PP^1\times \PP^1$ as a curve of
bidegree $(2,3)$. In this case there exists twelve $\bar
\FF_p$-points $p_1,..., p_{12}$ on $C$ such that if $X$ is the
blow up of $\PP^1\times \PP^1$ at these points then the line
bundle $L$, associated to the strict transform $\ti C$ of $C$, has
order $p$ after restricting to $\ti C$ but the restriction of
$L^{\otimes p}$ to $2\ti C$ is non-trivial. In this case Totaro
shows the following theorem (see \cite[proof of Theorem 6.1]{To}):

\begin{Theorem}
The line bundle $L$ is nef but it is not semiample. In fact, we
have $h^0(X,L^n)=1$ for every positive integer $n$.
\end{Theorem}

\end{Example}

\medskip

Totaro used the above theorem to show the first example of nef and
big line bundle on a smooth projective threefold, which is not
semiample. This shows that Artin's theorem does not generalize to
higher dimensions. These examples do not answer the following
question of Keel {(see \cite[Question 0.9]{Ke2})}, which we
provocatively formulate as a conjecture:

\begin{Conjecture} \label{Keel's-question}
Let $L$ be a strictly nef line bundle on a smooth projective
surface $X$ defined over $\bar\FF _p$. Then $L$ is ample.
\end{Conjecture}

By the Nakai-Moishezon criterion (see \cite[Chapter V, Theorem
1.10]{Ht}), or by Theorem \ref{Artin}, this conjecture is
equivalent to non-existence of strictly nef line bundles $L$ on
$X$ with $L^2=0$. In fact, in view of Totaro's example, one can
pose an even stronger conjecture:

\begin{Conjecture} \label{nef-effective}
Let $L$ be a nef line bundle on a smooth projective surface $X$
over $\bar\FF _p$.  Then the Iitaka dimension $\kappa (L)$ of $L$
is non-negative. Equivalently, we can find some positive integer
$m$ such that $L^{\otimes m}$ has a section.
\end{Conjecture}

If $L$ is nef and $L^2>0$ then $\kappa (L)=2$, so in the above
conjecture we can assume that $L^2=0$. We can also try to relax
the nefness assumption and pose the following conjecture:

\begin{Conjecture} \label{pseudoeffective-effective}
Let $D$ be a pseudoeffective $\QQ$-divisor on a smooth projective
surface $X$ over $\bar\FF _p$. Then $D$ is $\QQ$-linearly
equivalent to an effective $\QQ$-divisor.
\end{Conjecture}

Conjecture \ref{pseudoeffective-effective} is equivalent to
non-existence of a nef line bundle $L$ with Iitaka dimension
$\kappa (L)=-\infty$ and the numerical Iitaka dimension $\nu
(X)=1$. Obviously, all of the above conjectures can be also
considered in higher dimensions but similarly to the surface case
no answer seems to be known up to date. In fact, in higher
dimensions Conjecture \ref{pseudoeffective-effective} can be
generalized into two different ways: either as asking wether the
cone of curves $\NE (X)\subset N_1(X)_{\QQ}$ is closed or as
asking wether the cone of effective divisors is closed.

\medskip

The assertion of Conjecture \ref{pseudoeffective-effective} seems
to be much stronger than the one of Conjecture \ref{nef-effective}
but in fact we have the following lemma:

\begin{Lemma} \label{nef=eff-iff-pseudo=eff}
Conjectures \ref{nef-effective} and
\ref{pseudoeffective-effective} are equivalent.
\end{Lemma}

\begin{proof}
We only need to check that Conjecture \ref{nef-effective} implies
Conjecture \ref{pseudoeffective-effective}.

Let $D$ be  a pseudoeffective $\QQ$-divisor. Then there exists a
decomposition (so called \emph{Zariski decomposition}) $D=P+N$,
where $P$ is a nef $\QQ$-divisor and $N$ is a (negative) effective
$\QQ$-divisor $N$ such that $P\cdot N=0$. By our assumption we
know that some positive multiple of $P$, and therefore also of
$D$, has a section.
\end{proof}

\section{Killing cohomology by finite morphisms}

If $\cL$ is an ample line bundle on a smooth variety $X$ defined
over a field of characteristic zero then Kodaira's vanishing
theorem says that $H^i(X, \cL ^{-1})$ vanishes for $i< \dim X$.
Kawamata--Viehweg vanishing theorem says that the same vanishing
holds if $\cL$ is only nef and big. However, Raynaud in \cite{Ra}
constructed an example showing that already Kodaira's vanishing
theorem fails in positive characteristic. In this section we do
not try to recover Kodaira's vanishing theorem adding additional
assumptions on the base variety as was done by Deligne and Illusie
in \cite{DI}. Instead try to kill cohomology on all varieties but
using finite morphisms:

\begin{Theorem}\label{Bhatt-vanishing}
Let $X$ be a proper variety over a field of positive
characteristic and let $\cL$ be a semiample line bundle on $X$.
\begin{enumerate}
\item For any $i>0$ there exists a finite surjective morphism $\pi : Y\to
X$ such that the induced map $H^i(X, \cL) \to H^i (Y, \pi^* \cL)$
is zero.
\item If $\cL$ is big then for any $i<\dim X$ there exists
a finite surjective morphism $\pi : Y\to X$ such that the induced
map $H^i(X, \cL ^{-1}) \to H^i (Y, \pi^* \cL^{-1})$ is zero.
\end{enumerate}
\end{Theorem}

This theorem was proven by Hochster and Huneke \cite[Theorem
1.2]{HH} in case $\cL$ is a tensor power of a very ample line
bundle (see also \cite[Theorem 2.1]{Sm} and its errratum for the
case when $\cL$ is a tensor power of an ample line bundle), and by
Bhatt \cite[Propositions 7.2 and 7.3]{Bh} in general. Note that in
case $X$ is Cohen--Macaulay and $\cL$ is a tensor power of an
ample line bundle, then the only non-trivial case is when $\cL=\cO
_X$. In the remaining cases, it is sufficient to use Serre's
vanishing theorem (see \cite[Chapter III, Theorem 5.2]{Ht}) and
Serre's duality (see \cite[Chapter III, Corollary 7.7]{Ht}) in the
dual case.

One can ask wether Theorem \ref{Bhatt-vanishing} works under
weaker assumptions on $\cL$, possibly after restricting the base
field to the algebraic closure of a finite field (this is the most
interesting case, as it is the only case that arises when reducing
from characteristic zero). By Proposition \ref{torsion-lb},
Theorem \ref{Bhatt-vanishing}.1 holds for nef line bundles on
curves over $\bar \FF _p$ but it fails for nef line bundles on
smooth projective surfaces over $\bar \FF _p$. More precisely, one
can prove that in Example \ref{Tot-example} we have the following
non-vanishing theorem (see \cite[Theorem 3.1]{La2}):

\begin{Theorem}\label{Totaro}
Let $M=L^{-p-1}$ or $M=L^{p-1}$. Then for any complete $\bar
\FF_p$-surface $Y$ and any generically finite surjective morphism
$\pi : Y\to X$ the induced map $H^1(X, M)\to H^1(Y, \pi^*M)$ is
non-zero.
\end{Theorem}

Similarly, Theorem \ref{Artin} implies that Theorem
\ref{Bhatt-vanishing}.1 holds for nef and big line bundles on
smooth projective surfaces over $\bar \FF _p$ but one can show
that it fails for nef and big line bundles on smooth projective
threefolds over $\bar \FF _p$ (see \cite[Proposition 4.1]{La2}).

\medskip

In analogy to the Kawamata--Viehweg vanishing theorem, it is more
natural to generalize Theorem \ref{Bhatt-vanishing}.2 to nef and
big line bundles on smooth projective varieties. In fact, in low
dimensions one can show an even stronger theorem:

\begin{Theorem} \label{easy-vanishing}
Let $L$ be a nef and big line bundle on a normal projective
variety over field of positive characteristic. Fix an integer
$0\le i< \min (\dim X, 2)$. Then for sufficiently large $m$ the
map
$$H^i(X, L^{-1})\to H^i(X, L^{-p^m})$$ induced by the $m$-th Frobenius pull back is
zero.
\end{Theorem}

Unfortunately, the vanishing holds for trivial reasons because
under the above assumptions one has $H^i(X, L^{-n})=0$ for $n\gg
0$ (see \cite[Theorem 10]{Fu}; see also \cite[Theorem 2.22 and
Corollary 2.27]{La1} for effective versions of this theorem).

The only known examples of nef and big line bundle $L$ on a smooth
projective variety $X$ of dimension $>2$ such that $H^2(X,
L^{-n})\ne 0$ for all $n\gg 0$ were constructed by Fujita (see
\cite[pp. 526--527]{Fu}). He used Raynaud's counterexample to
Kodaira's vanishing theorem in positive characteristic (see
\cite{Ra}). By construction, in Fujita's example the map induced
by the $m$-th Frobenius pull back on $H^2(X, L^{-1})$ vanishes for
all $m\gg 0$. This leaves open the following question:

\begin{Question} \label{interesting}
Let $L$ be a nef and big line bundle on a smooth projective
variety defined over an algebraically closed field of positive
characteristic. Fix an integer $0\le i< \dim X$. Is the map
$H^i(X, L^{-1})\to H^i(X, L^{-p^m})$ induced by the $m$-th
Frobenius pull back zero for $m\gg 0$?
\end{Question}

Note that \cite[Example 5.4]{La2} shows that the answer to this
question is negative if one allows singular varieties. But for
smooth varieties an answer to the above question is not known even
if $L$ is semiample and big.

\medskip
One can also try to weaken conditions on $L$ in Theorem
\ref{easy-vanishing} still hoping that we can kill cohomology
using the Frobenius morphism. This works in some cases as shown by
the following theorem proven in \cite[Theorem 6.1]{La2}:

\begin{Theorem}\label{surface-vanishing}
Let $X$ be a smooth projective surface defined over an algebraic
closure of some finite field. Let $L$ be a nef line bundle on $X$
such that $\kappa(X, L)=-\infty$ (i.e., no power of $L$ has any
sections). Then for large $n$ the map $H^1(X, L^{-1})\to
H^1(X,(F_X^n)^*L^{-1})$ induced by the $n$-th Frobenius morphism
$F_X^n$ is zero.
\end{Theorem}

Note that if in Example \ref{Tot-example} we take $M=L^{p+1}$ then
we get a nef line bundle with $M^2=0$ on a smooth projective
surface over $\bar \FF_p$ such that $H^1(X, M^{-1})\to
H^1(X,M^{-p^n})$ induced by the $n$-th Frobenius pull back is
always non-zero (see Theorem \ref{Totaro}).

The above theorem is consistent with Conjecture
\ref{nef-effective} saying that there does not exist a nef line
bundle $L$ on a smooth projective surface defined over $\bar \FF
_p$ such that $\kappa (L)=-\infty$ (cf. Corollary
\ref{surface-reduction-of-strictly-nef}).

An interesting point in proof of Theorem \ref{surface-vanishing}
is that we use the higher rank case of Proposition
\ref{torsion-lb}, which follows from boundedness of the family of
semistable vector bundles with trivial Chern classes.

As a corollary to Theorem \ref{surface-vanishing} we get the
following theorem analogous to Theorem \ref{easy-vanishing}:

\begin{Corollary}\label{higher-dim-vanishing}
Let $X$ be a smooth projective variety of dimension $d\ge 2$
defined over an algebraic closure of some finite field. Let $L$ be
a strictly nef line bundle on $X$. Then for large $n$ the map
$H^1(X, L^{-1})\to H^1(X,(F_X^n)^*L^{-1})$ induced by the $n$-th
Frobenius morphism $F_X^n$ is zero.
\end{Corollary}

\section{Vanishing theorems in mixed characteristic}

Let $R$ be a domain which contains $\ZZ$ and which, as a ring, is
finitely generated over $\ZZ$.  Let $\cX$ be a projective
$R$-scheme and let $\cL$ be an invertible sheaf of $\cO
_{\cX}$-modules. Let $\cX_s$ denote the fibre over $s\in S$ and
let $\cL_s$ be the restriction (i.e., pull-back) of $\cL$ to
$\cX_s$.

Let $R\subset K$ be an algebraic closure of the field of quotients
of $R$. By assumption $K$ is of characteristic zero, so we can
think of $\cX\to S=\Spec R$ as a  model of the generic geometric
fibre $\cX_K$ with polarization $\cL _K$.

The following theorem (see \cite[3.5]{Sm}), conjectured by Huneke
and K. Smith in \cite[3.9]{HS}, was proven (in more general
setting of rational singularities) by N. Hara in \cite[Theorem
4.7]{Ha} and later by V. Mehta and V. Srinivas in \cite[Theorem
1.1]{MSr}.

\begin{Theorem}
Let us assume that $\cX_K$ is smooth and $\cL _K$ is ample. Then
there exists a non-empty Zariski open subset $U\subset S$ such
that for every closed point $s\in U$ the natural map
$$H^i(\cX_s, \cL _s^{-1})\to H^i(\cX_s, F^*\cL _s^{-1}), $$
induced by the Frobenius morphism on the fiber $\cX_s$, is
injective for all $i\ge 0$.
\end{Theorem}

Note that for $i< \dim \cX _K$ Kodaira's vanishing theorem says
that $H^i(\cX_K, \cL _K^{-1})=0$ so by semicontinuity of
cohomology (see \cite[III, Theorem 12.8]{Ht}) we have $H^i(\cX_s,
\cL _s^{-1})=0$ for $s$ from some open subset of $S$. So the above
theorem is non-trivial only in case $i=\dim X$. On the other hand,
one can ask if similar theorems hold in other cases when we do not
have vanishing of cohomology at the generic fibre. Here is one
such example in the surface case:

\begin{Proposition} \label{mixed-surface-vanishing}
Let us assume that $\cX_K$ is a smooth surface and  $\cL _K$ is a
line bundle with $\kappa (\cL _K)=-\infty$. Assume also that there
exists an ample line bundle $\cA _K$ on $\cX_K$ such that $c_1\cL
_K \cdot c_1 \cA_K> 0$. Then there exists a non-empty Zariski open
subset $U\subset S$ such that for every closed point $s\in U$ and
every positive integer $n$ the natural map
$$H^1(\cX_s, \cL _s^{-1})\to H^1(\cX_s, (F^n)^*\cL _s^{-1}), $$
induced by composition of $n$ absolute Frobenius morphisms on the
fiber $\cX_s$, is injective.
\end{Proposition}

\begin{proof}
Let $B^1_{\cX_s}$ be the sheaf of exact $1$-forms. By definition
we have an exact sequence
$$0\to \cO_{\cX_s}\to F_*\cO_{\cX_s}\to F_*B^1_{\cX_s}\to 0.$$
Therefore to check that
$$H^1(\cX_s, \cL _s^{-1})\to H^1(\cX_s, F_*\cO_{\cX_s} \otimes \cL _s^{-1})=H^1(\cX_s, F^*\cL _s^{-1}) $$
is injective, it is sufficient to prove that $H^0(\cX_s,
F_*B^1_{\cX_s}\otimes \cL_s^{-1})=0$. But $F_*B^1_{\cX_s}$ is a
subsheaf of $F_*\Omega_{\cX_s}^1$, so by the projection formula we
have
$$H^0(F_*B^1_{\cX_s}\otimes \cL_s ^{-1})\subset  H^0(F_*\Omega_{\cX_s} \otimes \cL_s ^{-1})=H^0(\Omega_{\cX_s} \otimes
F^*\cL_s ^{-1}).$$ So it is sufficient to show that there exists
an open subset $U\subset S$ such that for every closed point $s\in
U$ the sheaf $\Omega_{\cX_s} \otimes F^*\cL_s ^{-1}$ has no
sections. Similarly, to check that
$$H^1(\cX_s, (F^{n-1})^*\cL _s^{-1})\to H^1(\cX_s, (F^n)^*\cL _s^{-1})$$
is injective it is sufficient to prove that $\Omega_{\cX_s}
\otimes (F^{n})^*\cL_s ^{-1}$ has no sections.

We can find a Zariski open subset $V\subset S$ and a line bundle
$\cA$ extending $\cA _K$. Since ampleness is an open property,
shrinking $V$ if necessary, we can assume that $\cA$ on $\cX_V\to
V$ is relatively ample. Existence of the relative
Harder-Narasimhan filtration of $\Omega _{\cX_V/V}$ (see
\cite[Theorem 2.3.2]{HL}) implies that further shrinking $V$ we
can assume that for all closed points $s\in V$ we have
$$\mu _{\max ,\cA_s}(\Omega_{\cX_s})=\mu _{\max , H}(\Omega_{\cX_K}).$$
Since $c_1\cL _s\cdot c_1\cA_s =c_1\cL _K \cdot c_1\cA _K > 0$, we
see that if the characteristic $p$ at a closed point $s\in V$ is
larger than $\mu_{\max , \cA}(\Omega_{\cX_k})/(c_1\cL _K \cdot
c_1\cA _K)$, then for every positive integer $n$
$$\mu _{\max ,\cA_s}(\Omega_{\cX_s} \otimes
(F^n)^*\cL_s ^{-1})=\mu _{\max ,\cA_K}(\Omega_{\cX_K})-p^n (c_1\cL
_K \cdot c_1\cA _K)<0.$$ But existence of sections of $\Omega_{\cX
_s} \otimes (F^{n})^*\cL_s ^{-1}$ would contradict this
inequality.
\end{proof}

\medskip

\begin{Lemma} \label{pseudo-criterion}
Let $C$ be a $\QQ$-divisor on a smooth projective surface $X$. If
$C^2\ge 0$ and $CP>0$ for some nef divisor $P$ then $C$ is
pseudoeffective.
\end{Lemma}

\begin{proof}
If $CA<0$ for some ample divisor $A$ then taking appropriate
combination $H=aA+bP$ for some $a,b>0$ we have $CH=0$. Since $H$
is ample and $C$ is numerically non-trivial, the Hodge index
theorem (see \cite[Chapter V, Theorem 1.9]{Ht}) gives $C^2<0$.
\end{proof}

\medskip

\begin{Corollary} \label{surface-reduction-of-strictly-nef}
Let $L$ be a pseudoeffective line bundle on a smooth projective
surface defined over a field of characteristic zero. Let us assume
that $L^2\ge 0$ and $H^1(X, L ^{-1})$ is non-zero. Then for almost
all primes $p$ the reduction of $L$ modulo $p$ has a non-negative
Iitaka dimension.
\end{Corollary}

\begin{proof}
If $L$ is pseudoeffective then and $L^2\ge 0$ then by Lemma
\ref{pseudo-criterion} almost all reductions of $L$ are
pseudoeffective. Let $L=P+N$ be the Zariski decomposition (see
proof of Lemma \ref{nef=eff-iff-pseudo=eff}). If $L$ is not nef
then $P^2=L^2-N^2>0$ (since $N$ is non-zero we have $N^2<0$ as
follows from $PN=0$ by the Hodge index theorem). Hence $P$ is big,
which implies that $L$ is also big. The same argument shows that
if we take a reduction of $L$ which is pseudoeffective but not nef
then it is big. So we can assume that a reduction of $L$ is nef.
In this case the assertion follows from Proposition
\ref{mixed-surface-vanishing} and Theorem \ref{surface-vanishing}.
\end{proof}

\begin{Remarks}

\begin{enumerate}
\item In the above corollary, instead of assuming that $H^1(X, L
^{-1})$ is non-zero it is sufficient to assume that there exist a
smooth projective surface $Y$ and a generically finite morphism
$\pi:Y\to X$ such that $H^1(Y, \pi^*L ^{-1})$ is non-zero.
\item Corollary \ref{surface-reduction-of-strictly-nef} implies that if
a line bundle $L$ is strictly nef with non-vanishing $H^1(X,
L^{-m})$ for some positive integer $m$, then its reduction to
positive characteristic is almost never strictly nef. This
happens, e.g., in Mumford's example (see Example \ref{Mumford}).
 In fact, in
this case Biswas and Subramanian (see \cite[Theorem 1.1]{BiS})
proved that strictly nef line bundles on ruled surfaces over
$\bar\FF _p$ are always ample.
\end{enumerate}
\end{Remarks}

\section{Examples of strictly nef line bundles}

Note that if $L$ is a strictly nef line bundle on a proper variety
$X$ and $f: Y\to X$ is a finite morphism then $f^*L$ is also
strictly nef. This gives a lot of examples of strictly nef line
bundles once we have constructed some such bundles. In this
section we review known constructions of strictly nef line bundles
on smooth projective surfaces that do not come from this
construction.

\begin{Example} \label{Mumford}
The most famous example of a strictly nef line bundle is due to
Mumford (see \cite[I, Example 10.6]{Hart}). Namely, let $C$ be a
smooth complex projective curve of genus $\ge 2$. Then on $C$
there exists a rank $2$ stable vector bundle $E$ with  trivial
determinant and such that all symmetric powers $S^nE$ are also
stable. Let $\pi: X=\PP (E)\to C$ be the projectivization of $E$
and let $L=\cO _{\PP (E)}(1)$. Then $L$ is a strictly nef line
bundle on $X$ with $L^2=0$. Note that in this example
$H^1(X,L^{-2})$ is non-zero. More precisely, let us not that the
relative Euler exact sequence
$$0\to \Omega _{X/C}\to \pi ^*E\otimes L^{-1} \to \cO _X\to 0$$
is non split, as it is non-split after restricting to the fibers
of $\pi$. After tensoring this sequence by $L$ and using $\det
E\otimes \cO _C$ we get the sequence
$$0\to L^{-1}\to \pi ^*E \to L\to 0,$$
which gives a non-zero element in $\Ext ^1(L, L^{-1})=H^1(X,
L^{-2})$.

For generalization of Mumford's example to higher dimensions see
S. Subramanian's paper \cite{Su}. For uncountable fields of
positive characteristic a similar example was considered  by V.
Mehta and S. Subramanian  \cite{MS}. The next example shows
existence of strictly nef line bundles even over countable fields
of positive characteristic, provided they have sufficiently large
transcendental degree over its prime field.
\end{Example}

\medskip

\begin{Example} \label{Nagata}
Consider the projective plane $\PP^2$ over some field $k$ and let
us take $r=s^2$, where $s>3$, $k$-rational points $p_1,..., p_r\in
\PP^2(k)$. Let $p: X\to \PP^2$ be the blow up at these points and
let us take $L=p^*\cO_{\PP^2} (s)\otimes \cO(-E)$, where $E$ is
the exceptional divisor of $p$. Clearly, we have $L^2=0$. If all
the chosen points lie on a geometrically irreducible degree $s$
curve $C\subset \PP^2$ defined over $k$ then $L$ is nef. This
follows from the fact that the strict transform $\ti C$ gives an
element of the linear system $|L|$ and hence for every irreducible
curve $D\subset Y$ we have $D\cdot L=D\cdot \ti C \ge 0$ with
equality if and only if $D=\ti C$. This is also the main idea
behind Totaro's construction of a nef non-semiample line bundle,
except that to obtain an example where $C$ has genus $2$ he blow
ups $\PP ^1\times \PP^1$ instead of $\PP ^2$. Obviously, the
bundle $L$ obtained in this way is not strictly nef as $L\cdot \ti
C=0$. However, Nagata proved the following theorem:

\begin{Theorem}
Assume that the points $p_1,..., p_r$ are very general. Then $L$
is strictly nef.
\end{Theorem}

\begin{proof}
Let $D$ be any reduced curve on the blow up $X$ and let $C\in |\cO
_{\PP ^2}(d)|$ be its image. Let $m_1,..., m_r$ be the
multiplicities of $C$ at the points $P_1,..., P_r$, respectively.
Then $LD=sd-\sum _{i=1}^rm_i$. But by \cite[Chapter 3, Proposition
1]{Na} we have $sd-\sum _{i=1}^rm_i>0$.
\end{proof}

\medskip

Unfortunately, this theorem does not say anything for varieties
defined over $\bar \FF _p $.

Note that a similar construction can be used also in different
cases: we can blow up some points  $p_1,..., p_r$ (where $r$ can
be arbitrary) on a smooth projective surface $X$ and take the pull
back of an ample line bundle on $X$ twisted by a suitable negative
combination of exceptional divisors, arranging this so that the
obtained line bundle has self intersection $0$. If the number $r$
of points is sufficiently large and the points are in a very
general position then the obtained line bundle should be strictly
nef. This type of construction was used, e.g., in \cite[Example 3.3]{LR}
but it seems that the proof of strict nefness of the
obtained divisor is incorrect.
\end{Example}

\medskip

\begin{Example} \label{Shimura}
Let $F$ be a real quadratic field and let $D$ be a totally
indefinite quaternion $F$-algebra. Let us recall that a quaternion
algebra over $F$ is an $F$-algebra $D=F+Fi+Fj+Fij$ given by
$i^2=a$, $j^2=b$ and $ij=-ji$, where $a,b\in F$ are some non-zero
elements. $D$ is totally indefinite, if for both embeddings
$F\hookrightarrow \RR$ we have $\RR \otimes _FD\simeq M_2(\RR)$.
In this case we get two inequivalent real representations $\rho_i:
D\to M_2(\RR)$, $i=1,2$. On the algebra $D$ we can introduce a
\emph{norm} $N: D\to F$ by
$$N(x_0+x_1i+x_1j+x_2ij)=x_0^2-ax_1^2-bx_2^2+abx_3^2$$ for $x_i\in
F$. Let $\ti G$ be the group of elements of norm $1$ in a fixed
maximal order $R$ in $D$ and let $G=\ti G/\langle \pm 1\rangle$.
Let $\HH$ be the complex upper half plane. The group $G$  acts on
the product $\HH \times \HH$ by
$$\lambda (z_1,z_2)=(\rho_1(\lambda) z_1,\rho_2(\lambda) z_2).$$
In case $D$ is a division algebra, the quotient surface $X=\HH
\times \HH/G$ is compact. Let us also assume that $X$ is smooth
(all these assumptions are satisfied in some cases). Let $p_1,p_2:
\ti X=\HH \times \HH\to \HH$ be the two projections. Then
$\Omega^1_{\ti X}\simeq p_1^*\Omega^1_{\HH} \oplus
p_2^*\Omega^1_{\HH}$ as $G$-linearized bundles. So by descent we
have $\Omega_X^1\simeq L\oplus M$ for some line bundles $L$ and
$M$.   Then we have the following lemma:

\begin{Lemma} \emph{(\cite[Lemma 3]{SB1})}
The line bundles $L$ and $M$ are strictly nef with $L^2=M^2=0$.
\end{Lemma}

\begin{proof}
Let $C$ be a reduced and irreducible curve in $X$ and let $\ti C$
be an irreducible component of its pre-image in $\ti X$. The line
bundle $L|_C$ is represented by a form whose pull-back to $\ti C$
is the pull-back of a positive form from $\HH$. Therefore $CL=\deg
L|_C>0$. This shows that $L$ is strictly nef and in particular
$L^2\ge 0$. If $L^2>0$ then $L$ is ample by the Nakai--Moishezon
criterion (see \cite[V, Theorem 1.10]{Ha}). But by Bogomolov's
vanishing theorem $\Omega^1_X$ does not contain any ample
subbundles. Therefore $L^2=0$. The same proof works also for $M$.
\end{proof}

\cite{SB1} contains a more general example of the same type but we
will need this particular case later on (see Example
\ref{EST-example}).
\end{Example}

\section{Variation of positivity of line bundles}

It is known that ampleness is an open condition in families (not
necessarily flat). More precisely, let $S$ be an irreducible
noetherian scheme and let $\pi : {\cX}\to S$ be a proper morphism.
Let $\cL$ be a line bundle on $\cX$.

\begin{Theorem} \label{ampleness-is-open} \emph{(see \cite[III, Theorem 4.7.1]{Gr})}
If $\cL_{s_0}$ is ample on $\cX_{s_0}$ for some point $s_0\in S$
then $\cL _s$ is ample for a general point of  $S$, i.e., there
exists an open neighborhood $U\subset S$ of $s_0$ such that $\cL
_s$ is ample on $\cX_s$ for all $s\in U$.
\end{Theorem}

\begin{Corollary}
If $\cL_{s_0}$ is nef on $\cX_{{s_0}}$ for some geometric point
$s_0\in S$ then $\cL _{{s}}$ is nef for a very general point of
$S$, i.e., there exist countably many open and dense subsets
$U_m\subset S$ such that $\cL _{{s}}$ is nef for every geometric
point $s\in \bigcap U_m$.
\end{Corollary}

\begin{proof}
Using Chow's lemma we can reduce to the case where $\pi$ is
projective. Let $\cO_{\cX} (1)$ be a $\pi$-ample line bundle on
$\cX$. By Theorem \ref{ampleness-is-open} we know that for every
positive integer $m$ the set $U_m$ of points for which
$(\cL^{\otimes m}\otimes \cO_{\cX}(1))_s$ is ample is open and
dense in $S$. It is easy to see that these sets satisfy the
required assertion.
\end{proof}

\medskip

Note that we can assume that the sequence $\{ U_m\} _{m\in \NN}$
is descending, i.e., $U_{m+1}\subset U_m$ for all $m$ and one can
ask if such a sequence must stabilize. In general, this is too
much to hope for but $\bigcap U_m$ contains the generic geometric
point of $S$ so we can ask if it contains any closed points. This
is interesting only if $S$ has only countably many points as only
then the set of closed geometric points $s\in S$ for which $L_s$
is nef can be empty. Indeed, this can really happen as shown by
the following example due to Monsky \cite{Mon1}, Brenner \cite{Br}
and Trivedi \cite{Tr}:

\begin{Example}
Let us start with recalling the following result of Monsky \cite[Theorem]{Mon1}:

\begin{Theorem}
Let $R_t = K_t[x, y, z]/(P_t)$, where  $K_t$ is an algebraic
closure of ${\FF}_2(t)$ and set
$$P_t=z^4+xyz^2+ x^3z+y^3z+t x^2y^2.$$
Then the Hilbert-Kunz multiplicity of $R_t$ is equal to
$3 + 4^{-m(t)}$, where
$$m(t) =\left\{\begin{array}{l}
\hbox{degree of $\lambda$ over $\FF_2$, if $t=\lambda^2+\lambda$ is algebraic over }\FF_2,\\
\infty , \hbox{ if $t$ is transcendental over }\FF_2.\\
\end{array}\right.$$
\end{Theorem}

Now let $k=\overline{\FF}_2$ and let us  set $S=\AA^1_k$ with
coordinate $t$ and $\PP^2_k$ with homogeneous coordinates
$[x:y:z]$. Let $Y\subset \PP^2\times _k S$ be given by
$$z^4+xyz^2+ x^3z+y^3z+t x^2y^2=0$$
and let $\cE=p_1^*\Omega_{\PP^2}$, where $p_1:Y\to \PP^2$ is the
canonical projection. Consider the projection $p_2: Y\to S$. Then
$\cE _s$ is not strongly semistable for every closed point $s\in
S$ (even on the singular fiber over $0\in S$) but $\cE_{\eta}$
is strongly semistable for the generic point $\eta\in S$.
This follows from Monsky's theorem and the computation of the Hilbert-Kunz
multiplicity of $R_t$ in terms of strong Harder--Narasimhan filtration of
bundles $\cE_s$ for $s\in S$ due to Brenner \cite[Theorem 1]{Br} and Trivedi
\cite[Theorem 5.3]{Tr}. This computation implies that $\cE_s$ is strongly semistable
for $s: \Spec K_t\to S$ if and only if the Hilbert-Kunz multiplicity of $R_t$
is equal to $3$.

Let $\cX$ be the projectivization of $\cF= p_1^*(S^2\Omega_{\PP^2}
(1))$ over $Y$. Let $\cL=\cO _{\PP (\cF)}(1)$ and let $\pi: \cX\to
S$ be the composition of the projections $\cX\to Y$ and $p_2:Y\to
S$. Then $\cL_{\overline \eta}$ is nef for a generic geometric
point $\overline\eta \in S $ but $\cL _{{s}}$ is not nef for every
closed geometric point $s\in S$.

One can also show a similar example in equal characteristic $3$
(see \cite{Mon2}).
\end{Example}

\medskip
Note that in the above example $S$ was defined over an algebraic
closure of a finite field. It seems to be unknown if similar
examples can occur for $S$ defined over a countable field of
positive characteristic containing transcendental elements over
its prime field, or even in case $S$ is defined over
$\overline{\QQ}$. One might expect that the strange behavior of
variation of nefness in positive equal characteristic cannot occur
in mixed characteristic:\footnote{Recently the author constructed a counterexample to this conjecture (see \cite{La-Adv}).
But the conjecture can still be true under appropriate assumptions, e.g., if we require that
the rank of the Neron-Severi group stays the same on the fibers of $\pi$.
}

\begin{Conjecture} \label{anty-Totaro}
Let $R$ be a finitely generated integral domain over $\ZZ$,
containing $\ZZ$. Let $\pi : {\cX}\to S=\Spec R$ be a smooth
proper morphism.  Let $\cL$ be an invertible sheaf of $\cO
_{\cX}$-modules and assume that the restriction of $\cL$ to the
generic geometric fibre of $\pi$ is nef. Then the set $T$ of
closed points $s\in S$ such that $\cL _s$ is semiample is dense in
$S$.
\end{Conjecture}

Totaro's Example \ref{Tot-example} comes from characteristic zero
by reduction modulo $p$. The above conjecture suggests that such
examples are rather rare and almost all reductions of a fixed nef
line bundle are semiample.

Conjecture \ref{anty-Totaro} generalizes \cite[Problem 5.4]{Mi}
which considers the same question
in case $\cX$ is a projectivization of a rank $2$ vector bundle
over a curve (in this case if $s\in S$ is a closed point then
nefness of $\cL_s$ implies its semiampleness by the Lange--Stuhler
theorem; see Proposition \ref{La-St}).

Note that one can show examples in which the set $T$ is not open
in the set of closed points of $S$. The first such examples come
from an unpublished work \cite{EST} of Ekedahl, Shepherd--Barron
and Taylor:

\begin{Example} \label{EST-example}
Consider $X$ from Example \ref{Shimura}. The line subbundle
$L^{-1}\subset T_X\simeq L^{-1}\oplus M^{-1}$ defines a foliation.
If we reduce $X$ modulo some prime of characteristic $p$ then the
$p$-curvature map $L_p^{\otimes (-p)}=F^*(L_p^{-1})\to
T_{X_{p}}/L_p^{-1}=M_p^{-1}$, given by taking the $p$-th power of
a derivation, is $\cO_{X_{p}}$-linear. If $p$ is inert in $F$ then
this map is non-zero (see \cite[p.~23]{EST}). In this case we get
a section of $L_p^{\otimes p}\otimes M_p^{-1}$ and, similarly, we
get a section of $M_p^{\otimes p}\otimes L_p^{-1}$. Note that
$L_p$  is not nef (and hence it is not semiample). Otherwise, we
would have $-LM=L_p(pL_p-M_p)\ge 0$, whereas $LM>0$. Since $L_p$
is pseudoeffective and  $L_p^2=0$, existence of the Zariski
decomposition of $L_p$ implies that $L_p$ is big (see proof of
Corollary \ref{surface-reduction-of-strictly-nef}). Let us recall
that by Chebotarev's density theorem the number of rational primes
$p$ which remain inert in $F$ is infinite (of Dirichlet density
$1/2$). So in this case we have a strictly nef line bundle $L$ for
which infinitely many reductions are not semiample.

In fact, it is not clear how to prove that in the remaining cases
the reduction of $L$ is semiample (possibly apart from finitely
many primes).
\end{Example}

\medskip

Other examples of a similar type were obtained by Brenner
\cite{Br1} in case $\cX$ is a projectivization of a rank $2$
vector bundle over a curve (note that these examples did not solve
Miyaoka's problem \cite[Problem 5.4]{Mi}).

\section{Variation of semistability of  vector bundles} \label{Section:variation}

Let $X$ be a smooth complex projective variety and let $\cO _X(1)$
be an ample line bundle on $X$. Let $E$ be a slope semistable
(with respect to $\cO _X(1)$) locally free $\cO_X$-module.

We are interested in behavior of $E$ when taking reduction modulo
$p$. More precisely, all of the above data can be described by a
finite number of equations. Therefore there exist a subring
$R\subset \CC$, finitely generated as an algebra over $\ZZ$, and a
triple $(\cX, \cO_{\cX}(1), \cE)$ consisting of a smooth
projective $R$-scheme $\pi : \cX\to S=\Spec R$, an $R$-ample line
bundle $\cO_{\cX}(1)$ and a family $\cE$ of locally free slope
semistable sheaves on the fibers of $\pi$, such that on the fiber
over the generic geometric point $\Spec \CC\to S$ we recover the
triple $(X,\cO_X(1),E)$. Note that we have implicitly used
openness of slope semistability in flat families of sheaves.

Let us recall that for every maximal ideal $m\subset R$ the
residue field $k=R/m$ is finite of characteristic $p>0$. Now we
would like to relate various properties of $E$ to the behavior of
its reductions modulo $p$. We pose a series of conjectures that
should completely describe the behavior of strong semistability in
mixed characteristic. The first conjecture is motivated by
\cite{SB}, where it was proven in the rank $2$ case:

\begin{Conjecture}\label{Conjecture-SB}
Let $\Sigma^{nss}$ be the set of closed points $s\in S$ such that
$\cE_s$ is not strongly slope semistable. If $\Sigma^{nss}$ is
infinite \footnote{This assumption is tentative and works well only in the number field case. In general,
it should probably be modified so that the set $\Sigma^{nss}$
is dense in S. (Unfortunately, in the published version this footnote was misplaced.)}
then $\End E$ is a numerically flat vector bundle.
Moreover, $\End E$ is not \'etale trivializable.
\end{Conjecture}

\medskip

\begin{Lemma}
If $\Sigma^{nss}$ is infinite then $\End E$ is not \'etale
trivializable. In particular, Conjecture \ref{Conjecture-SB} is
true in the curve case.
\end{Lemma}

\begin{proof}
If $\End E$ is \'etale trivializable then $\End \cE$ is \'etale
trivializable over $\cX _U$ for some open subset $U\subset S$. In
particular, $\End {\cE}_s$, is strongly semistable for $s\in U$.
We claim that $\cE_s$ is also strongly semistable. If $\cE_s$ is
not strongly semistable then there exists some $n$ such that the
$n$th Frobenius pull back of $\cE_s$ is destabilized by some
subsheaf $E'$. But then
$$\mu (E'\otimes (F^n)^* \cE_s^*)= \mu (E')+
\mu ((F^n)^* \cE_s^*)>\mu ((F^n)^* \cE_s)+\mu ((F^n)^*
\cE_s^*)=0$$ and hence $E'\otimes (F^n)^* \cE^*_s$ destabilizes
$(F^n)^*(\End {\cE}_s)$, a contradiction. This implies that
$\Sigma^{nss}$ is contained in the set of closed points of $S-U$,
and therefore $\Sigma^{nss}$ is finite.

If $X$ is a curve then for every semistable $E$ the bundle $\End
E$ is semistable of degree $0$, so it is numerically flat and the
conjecture follows from the first part of the lemma.
\end{proof}

\medskip

This shows that Conjecture \ref{Conjecture-SB} is of interest only
in the surface case and the only non-trivial part of the
conjecture is that $\End E$ is numerically flat. Indeed, the
higher dimensional case can be easily reduced to the surface case
by means of restriction theorems. More precisely, if $X$ has
dimension $d$ greater than $2$ and $E$ is a vector bundle for
which $\Sigma^{nss}$ is infinite then the restriction of $E$ to a
general complete intersection surface $Y\subset X$ is semistable
and it satisfies the assumptions of the conjecture. So if we know
the conjecture for $E|_Y$ then $\End E|_Y$ is a numerically flat
vector bundle. But then  $\End E$ is also numerically flat because
it is semistable with respect to some ample polarization $H$ such
that $c_1(E)H^{d-1}=c_2(E)H^{d-2}=0$ (cf. \cite[Theorem 2]{Si}).

\section{Arithmetic of numerically flat vector bundles} \label{Section:arithmetic-of-flat-bundles}

Conjecture \ref{Conjecture-SB} implies that to study strong
semistability of reductions of a complex vector bundle, it is
sufficient to study reductions of numerically flat vector bundles.
The following subsection recalls a special role of such vector
bundles and their relation to representations of the fundamental
group.

\subsection{Flat bundles} \label{flat-bundles}
Let $X$ be a smooth complex projective variety. Giving a
representation of the topological fundamental group $\pi_1(X,x)$
on a complex vector space $V_x$ is equivalent to giving a complex
local system $V$ (a sheaf of complex vector spaces locally
isomorphic to the constant sheaf $\underline \CC ^n$, $n\in \NN$).
Given a local system we can recover the corresponding
representation as the \emph{monodromy representation}.

Given $V$ we can construct a holomorphic vector bundle
$\cO_X\otimes _{\CC}V$ with (holomorphic) integrable connection
$\nabla$ such that $\nabla (fv)=df\cdot v$, where $f$ is a local
section of $\cO _X$ and $v$ is a local section of $V$. On the
other hand, given a holomorphic vector bundle $\cE$ with
integrable connection $\nabla$ we can recover a local system $V$
as a sheaf of local sections $v$ of $\cE$ for which $\nabla
(v)=0$. This constructions provide functors giving an equivalence
of categories of complex local systems and holomorphic vector
bundles with integrable connection.

In \cite[Corollary 3.10]{Si} Simpson proved that these categories
are equivalent to the category of (Higgs) semistable Higgs bundles
$(E, \theta)$ with vanishing (rational) Chern classes. This
category contains the category of semistable vector bundles with
vanishing Chern classes. If a representation of $\pi_1(X,x)$ is an
extension of unitary representations, then the corresponding Higgs
bundle is an extension of stable vector bundles and the
equivalence preserves the holomorphic structure. In particular,
every semistable vector bundle with vanishing Chern classes has a
holomorphic flat structure which is an extension of unitary flat
bundles. Finally, let us recall that a vector bundle is semistable
with vanishing Chern classes if and only if it is numerically
flat.

\medskip

We also need to recall a few basic results about \'etale
trivializable bundles.

\subsection{\'Etale trivializable bundles}
Let $X$ be a smooth projective variety over an algebraically
closed field $k$. A  rank $r$ locally free sheaf $E$ on $X$ is
called \emph{\'etale trivializable} if there exists a finite
\'etale covering $\pi: Y\to X$ such that $\pi^*E\simeq \cO _Y^r$.
Over finite fields \'etale trivializable bundles are characterized
as Frobenius periodic bundles:

\begin{Proposition} \emph{(see \cite{LS})} \label{La-St}
Assume that $k=\bar \FF _p$ and let $F: X\to X$ be the Frobenius
morphism. A locally free sheaf $E$ is \'etale trivializable if and
only if there exists an isomorphism $(F_X^n)^*E\simeq E$ for some
positive integer $n$.
\end{Proposition}

It is easy to see that every \'etale trivializable bundle is
numerically flat. So we can try to characterize such bundles for
$k=\CC$ in terms of their monodromy representation. If we have a
representation $\rho: \pi_1(X,x) \to \GL _r(\CC)$ whose image $G$
is a finite group then by Weyl's trick $G$ is a unitary subgroup
of $\GL _r(\CC)$. Since every complex representation of a finite
group is a direct sum of irreducible representations, the
corresponding Higgs bundle $(E, \theta)$ is a direct sum of stable
vector bundles. Passing to the \'etale covering defined by the
quotient $\pi_1(X,x) \to G$ we see that each direct summand is
\'etale trivializable and the Higgs field $\theta=0$.

On the other hand, if a bundle is \'etale trivializable then it is
\'etale trivializable by a finite Galois covering and hence the
corresponding monodromy representation has finite image.

\subsection{\'Etale trivializability of reductions of numerically flat bundles}

We keep the notation from Section \ref{Section:variation} but now
we restrict to the case where $E$ is a numerically flat vector
bundle.

\begin{Conjecture}\label{Conjecture-et}
The set $\Sigma^{et}$ of closed points $s\in S$ such that $\cE_s$
is \'etale trivializable, is infinite.
\end{Conjecture}

The following example shows that this conjecture is interesting
even for very simple semistable vector bundles:

\begin{Example}\label{ordinary}
Let $X$ be a smooth complex projective variety with
$h^1(X,\cO_X)>0$. Let us consider vector bundle $E$ corresponding
to the extension
$$0\to  \cO_X\to E\to H^1(\cO_X)\otimes\cO_X\to 0$$
defined by the identity $\id _{H^1(\cO_X)}\in \End
(H^1(\cO_X))=\Ext ^1(H^1(\cO_X)\otimes\cO_X, \cO_X)$. This is
clearly a numerically flat vector bundle.

For every finite \'etale morphism $\pi: Y\to X$ the map $\pi^*:
H^1(\cO_X)\to H^1(\cO_Y)$ is injective as it can be split by the
trace map. Let $E_Y$ be the extension corresponding to $\id
_{H^1(\cO_Y)}\in \End (H^1(\cO_Y))=\Ext ^1(H^1(\cO_Y)\otimes\cO_Y,
 \cO_Y)$ and consider the commutative diagram
$$
\xymatrix{ 0\ar[r]&\cO_Y \ar[r]\ar@{=}[d] &\pi^* E\ar[r]\ar[d]&
H^1(\cO_X)\otimes \cO_Y\ar[r]\ar[d]^{\pi^*\otimes
\id_{\cO_Y}}&0\\
0\ar[r]&  \cO_Y\ar[r]& E_Y\ar[r]& H^1(\cO_Y)\otimes\cO_Y\ar[r]& 0.\\
}$$ If $\pi^*E$ is trivial then it injects into $E_Y$ and hence
$E_Y$ has at least $\rk E=h^1(\cO_X)+1$ linearly independent
global sections. But by the definition of $E_Y$ the connecting map
$H^0(Y, H^1(\cO_Y)\otimes\cO_Y )\to H^1(Y, \cO _Y)$ is an
isomorphism and hence $h^0(E_Y)=1$. Therefore $E$ is not \'etale
trivializable.

Let $\cX\to S$ be a model of $X$ as in the beginning of Section
\ref{Section:variation}.

\begin{Lemma}
There exists a non-empty open subset $U\subset S$ such that the
reduction $\cE_s$ of $E$ for a closed point $s\in U$ is \'etale
trivializable if and only if the Frobenius morphism $F=F_{\cX_s}$
acts on $H^1(\cO_{\cX_s})$ bijectively.
\end{Lemma}

\begin{proof}
If $F^*$ acts on $V=H^1(\cO_{\cX_s})$ bijectively then the diagram
$$
\xymatrix{ 0\ar[r]&\cO_{\cX_s} \ar[r]\ar@{=}[d] &F^*
\cE_s\ar[r]\ar[d]& V\otimes \cO_{\cX _s}\ar[r]\ar[d]^{F^*\otimes
\id_{\cO_{\cX _s}}}&0\\
0\ar[r]&  \cO_{\cX _s}\ar[r]& \cE_s\ar[r]& V\otimes\cO_{\cX _s}\ar[r]& 0\\
}$$ shows that $F^*\cE_s\simeq \cE_s$ and hence $\cE_s$ is \'etale
trivializable by the Lange--Stuhler theorem (see Proposition
\ref{La-St}).

Now assume that $\cE_s$ is \'etale trivializable. Let us consider
the unique decomposition $V=V_s\oplus V_n$ such that the Frobenius
morphism $F^*$ acts on $V_s$ as an automorphism and it is
nilpotent on $V_n$. Let $G$ be the bundle obtained as the
extension of  $V_n\otimes \cO_{\cX_s}$ by $\cO_{\cX_s}$ defined by
the canonical inclusion $(V_n\hookrightarrow V)\in \Hom (V_n,
V)=\Ext ^1(V_n\otimes\cO_{\cX_s}, \cO_{\cX_s})$. Then we have the
diagram
$$
\xymatrix{ 0\ar[r]& \cO_{\cX_s} \ar[r]\ar@{=}[d]& G\ar[r]\ar[d]&
V_n\otimes\cO_{\cX _s}\ar[r]\ar@{^{(}->}[d]&0\\
0\ar[r]&  \cO_{\cX _s}\ar[r]& \cE_s\ar[r]& V\otimes \cO_{\cX _s}\ar[r]& 0,\\
}$$ which shows that $G\hookrightarrow \cE_s$. By the definition
of $G$ there exists some $m_0$ such that $(F^{m_0})^*G$ is
trivial. Let $r=\dim V_n$. This shows that for every $m\ge m_0$ we
have
$$h^0((F^m)^*\cE_s)\ge h^0((F^m)^*G)=r+1.$$
By the Lange--Stuhler theorem we know that for some $m\ge m_0$ we
have $(F^m)^*\cE_s\simeq \cE_s$ and hence $h^0(\cE_s)\ge r+1$. By
the definition of  $E$ we know that the connecting map $\delta:
H^0(H^1(\cO_{X})\otimes \cO_{\cX _s})\to H^1(\cO_{X})$ is an
isomorphism and hence $h^0(E)=1$. Using semicontinuity of
cohomology, we see that there exists an open subset $U\subset S$
such that $h^0(\cE _s)=1$ for every $s\in U$. This implies that
for any closed $s\in U$ we have $r=0$ and $V=V_s$.
\end{proof}

Therefore Conjecture \ref{Conjecture-et} for vector bundle $E$ is
equivalent to the assertion that there are infinitely many closed
points $s\in S$ for which the Frobenius acts on $H^1(\cO_{\cX_s})$
bijectively. In the curve case this is equivalent to saying that
there are infinitely many places of ordinary reduction. This is
known in case of genus $g\le 2$ but it is still an open problem in
general.
\end{Example}

\medskip

\begin{Remark}
Note that if the reduction of $\cE _s$  is \'etale trivializable
by $\pi: Y\to \cX_s$ then the degree of $\pi$ is divisible by the
characteristic $p$ of the residue field $k(s)$. Indeed, if the
characteristic $p$ does not divide the degree of $\pi$ then
$\frac{1}{\deg \pi} \Tr _{\cX_s/Y}: \pi_*\cO_Y\to \cO_{\cX _s}$
splits the injection $\cO_{\cX_s}\to \pi_*\cO_Y$. Then the same
argument as in the characteristic zero case gives a contradiction.
\end{Remark}

\subsection{Analogue of the Grothendieck-Katz $p$-curvature conjecture}

In this subsection we try to relate \'etale trivializability of
reductions of a vector bundle to finiteness of the image of its
monodromy representation. Before formulating the corresponding
conjecture we provide its original motivation: the global case of
the Grothendieck--Katz conjecture.

Let $X$ be a smooth variety defined over a field of characteristic
$p>0$ and let $\nabla : E\to  \Omega_X \otimes E$  be an
integrable $k$-connection on a locally free $\cO _X$-module $E$.
In characteristic $p$, the $p$-th power $D^p$ of a derivation $D$
is again a derivation so we can consider $\nabla
(D^p)-\nabla(D)^p$. When this is zero for all local derivations
$D$ then we say that $\nabla$ has zero $p$-curvature. If $F_g:
X\to X^{(1)}$ is the geometric Frobenius morphism then $(E,
\nabla)$ is equivalent to giving a locally free
$\cO_{X^{(1)}}$-module $G$. The sheaf $G$ can be recovered from
$(E, \nabla)$ as a sheaf of local sections $v$ of $E$ for which
$\nabla (v)=0$. On the other hand, giving $G$ we can construct a
canonical connection on $E=F_g^*G$ by differentiating along the
fibers of $F_g$, i.e., we set $\nabla (f\otimes g)=df\otimes g$.

\begin{Conjecture} \emph{(Grothendieck--Katz, see \cite{Ka})}
Let $(E,\nabla)$ be a holomorphic vector bundle with an integrable
connection on a complex manifold $X$. Then $(E,\nabla)$ has a
finite monodromy group if and only if almost all its reductions to
positive characteristic have vanishing $p$-curvature.
\end{Conjecture}

Note that if $X$ projective then $(E,\nabla)$ with finite
monodromy group corresponds via Simpson's correspondence described
in Subsection \ref{flat-bundles} to an \'etale trivializable
bundle (with zero Higgs field). So we can try to describe
representations of the fundamental group with finite image on the
Higgs bundle side in the following way:

\begin{Conjecture}\label{Conjecture-net}
In the notation of Section \ref{Section:variation} assume that $E$
is not \'etale trivializable. Then the set $\Sigma^{net}$ of
closed points $s\in S$ such that $\cE_s$ is not \'etale
trivializable, is infinite.
\end{Conjecture}

In case of bundles described in Example \ref{ordinary}, the
conjecture can be reformulated as saying that for a given smooth
complex projective variety $X$ with $h^1(X,\cO_X)>0$, there are
infinitely many points $s\in S$ for which the nilpotent part of
the Frobenius action on $H^1(\cO_{\cX_s})$ is non-trivial. In
particular, if $X$ is a complex elliptic curve then this is
equivalent to saying that there are infinitely many primes for
which the reduction of $X$ is supersingular. In case of elliptic
curves defined over $\QQ$ (and also in some other cases) this is a
celebrated Elkies' result \cite{El}.

\medskip

\begin{Example}
Let $A$ be an abelian variety over a number field $K$ and let
$\cL$ be a line bundle on some model $\cA\to S=\Spec R$ of $A$ for
a finitely generated subring $R\subset K$. Note that by Theorem
\ref{La-St} a line bundle $L$ on a smooth projective variety over
$\bar \FF _p$ is \'etale trivializable if and only if there exists
some $n\in \NN$ such that $(F^n)^*L\simeq L$. Therefore Conjecture
\ref{Conjecture-net} predicts that in the above case if for almost
all closed points $s\in S$ there exists $n_s\in \NN$ such that
$(F^{n_s})^*\cL_s\simeq \cL_s$ then $\cL _K$ is \'etale
trivializable on $A$.

In this case a slightly weaker result is known. Namely, assume
that there exists some $n\in \NN$ such that for almost all closed
points $s\in S$ we have $(F^n)^*\cL_s\simeq \cL_s$ (so $n_s$ in
the above reformulation is independent of $s$). Then $\cL_K$ is
\'etale trivializable on $A$. This is just a dual version of
\cite[Theorem 5.3]{Pi} and it implies that Conjecture
\ref{Conjecture-net} reduces to existence of a uniform bound on
all $n_s$.

Note that if $\cL_K$ is \'etale trivializable then there exists a
positive integer $m$ such that $\cL_s ^m\simeq \cO _{\cX _s}$ for
all $s$ from some non-empty open subset $U\subset S$. Since for
every (rational) prime $p$ not dividing $m$ the number $p^{m!}-1$
is divisible by $m$ we see that $(F^{m!})^*\cL_s\simeq \cL_s$ for
all closed points $s$ from some smaller non-empty open subset
$V\subset U$. This provides us with the converse to Pink's
theorem.
\end{Example}

\medskip

Using the same methods as in proof of \cite[Th\'eor\`eme
7.2.2]{An} and \cite[Theorem 5.1]{EL} one can show that an
analogue of Conjecture \ref{Conjecture-net} holds in case of equal
characteristic zero:

\begin{Theorem}
Let $f: \cX\to S$ be a smooth projective morphism of varieties
defined over an algebraically closed field $k$ of characteristic
$0$. Let $\bar \eta$ be the generic geometric point of $S$ and let
$\cE$ be a locally free sheaf on $\cX$. Let us assume that there
exists a dense subset $U\subset S(k)$ such that for every $s$ in
$U$ the bundle $\cE_s$ is \'etale trivializable. Then we have the
following:
\begin{itemize}
\item[1)]  There exists a finite Galois \'etale covering $\pi
: Y\to \cX_{\bar \eta}$  such that $\pi ^*\cE_{\bar \eta}$ is a
direct sum of line bundles.
\item[2)] If $U$ is open in $S(k)$ then $\cE_{\bar \eta}$ is \'etale
trivializable.
\end{itemize}
\end{Theorem}

Note that, similarly as in other cases, an analogue of this
theorem is false for families defined over an algebraic closure of
a finite filed:

\begin{Example}
In \cite[Corollary 4.3]{EL} the authors used Laszlo's example
\cite[Section 3]{Ls}) to construct a locally free sheaf $\cE$ on
$\cX=X\times _k S\to S$, where $X$ is a smooth projective curve,
$S$ is a smooth curve, both defined over $k=\bar \FF _2$ and such
that for every closed point $s\in S$ the bundle $\cE_s$ is \'etale
trivializable but $\cE_{\bar \eta}$ is not \'etale trivializable
for the generic geometric point ${\bar \eta}$ of $S$.
\end{Example}

\medskip
The above example can occur only because the monodromy groups of
$\cE_s$ have orders divisible by the characteristic of $k(s)$. For
positive results in other cases see \cite[Theorem 5.1]{EL}.

\bigskip

\emph{\large \bf Acknowledgements.}

The author would like to thank D. R\"ossler and H. Esnault for
useful conversations related to Section
\ref{Section:arithmetic-of-flat-bundles}.
Author's work was partially supported by Polish National
Science Centre (NCN) contract number 2012/07/B/ST1/03343.

\end{document}